\documentclass[11pt,a4paper]{article}
\usepackage{amsmath,amssymb}
\usepackage{bm}
\usepackage{graphicx}
\usepackage{ascmac}
\usepackage{amsthm}
\usepackage{mathrsfs}
\usepackage{enumerate}
\usepackage{cite}
\usepackage{euscript}
\theoremstyle{plain}
\newtheorem{thm}{Theorem}[section]
\newtheorem{cor}[thm]{Corollary}
\newtheorem{prop}[thm]{Proposition}
\newtheorem*{thm*}{Theorem}
\newtheorem{fact}[thm]{Fact}
\newtheorem{lem}[thm]{Lemma}

\theoremstyle{definition}
\newtheorem{dfn}[thm]{Definition}
\newtheorem{exam}[thm]{Example}
\newtheorem{rem}[thm]{Remark}

\providecommand{\keywords}[1]
{
  \small	
  \textbf{\textit{Keywords---}} #1
}

\newcommand{\Hom}{\mathop{\mathrm{Hom}}\nolimits}
\newcommand{\Ker}{\mathop{\mathrm{Ker}}\nolimits}

\newcommand{\diag}{\mathop{\mathrm{diag}}\nolimits}
\newcommand{\ad}{\mathop{\mathrm{ad}}\nolimits}
\newcommand{\Resub}{\mathop{\mathrm{Re}}\nolimits}
\newcommand{\sltwo}{\mathfrak{sl}_2}
\newcommand{\Oinf}{ \tilde{\mathcal{O}} }
\newcommand{\muone}{\mu'}
\newcommand{\mutwo}{\mu''}
\newcommand{\lambdaone}{\lambda'}
\newcommand{\lambdatwo}{\lambda''}
\newcommand{\lambdathree}{\lambda'''}

\title{Branching problem of tensoring two Verma modules and its application to differential symmetry breaking operators}
\author{Reiji Murakami\thanks{Graduate School of Mathematical Science, The University of Tokyo}}
\date{\today}

\begin{document}
\maketitle
\begin{abstract}
Kobayashi-Pevzner discovered in [Selecta Math., 2016] that the failure of the multiplicity-one property in the fusion rule of Verma modules of $\sltwo$ occurs exactly when the Rankin-Cohen bracket vanishes, and classified all the corresponding parameters. In this paper we provide yet another characterization for these parameters, and give a precise description of indecomposable components of the tensor product. Furthermore, we discuss when the tensor products of two Verma modules are isomorphic to each other for semisimple Lie algebras $\mathfrak{g}$.

\end{abstract}

\keywords{branching laws, symmetry breaking operator, tensor product, Verma module, BGG category, Lie algebra}

\section{Introduction.}
\noindent


The object of this article is the tensor product of two Verma modules. Recently, a new phenomenon has been discovered in the $\sltwo$-case, including the failure of the multiplicity-freeness in the fusion rule for \textit{singular} integral parameters, see \cite[Sect. 9]{KP} for details. This article provides a precise description of indecomposable components of the tensor product in the $\sltwo$-case (Theorem \ref{strtheorem}), and proves that the failure of the self-duality of indecomposable components corresponds exactly to the multiplicity-two phenomenon (Theorem \ref{KPtheorem}).

As an introduction, we begin with a general result of the tensor product of two Verma modules. Let $\mathfrak{g}$ be a complex semisimple Lie algebra, $\mathfrak{h}$ a Cartan subalgebra, and $\mathfrak{b}$ a Borel subalgebra containing $\mathfrak{h}$. Let $M(\mu):=U(\mathfrak{g})\otimes_{U(\mathfrak{b})} \mathbb{C}_\mu$ be the Verma module for $\mu\in \mathfrak{h}^*$, see Section \ref{notation} for further notation.

\begin{thm}\label{generaltheorem}
Suppose that $\mu',\mu'',\nu',\nu''\in \mathfrak{h}^{*}$ satisfy $\mu'+\mu''=\nu'+\nu''$. If each of the two sets $\{\mu',\mu'' \}$ and $\{\nu',\nu'' \}$
contains at least one anti-dominant element, then one has an isomorphism as $\mathfrak{g}$-modules:
\begin{align*}
M(\mu')\otimes M(\mu'')\cong M(\nu')\otimes M(\nu'').
\end{align*}
\end{thm}

The anti-dominant assumption cannot be dropped, as we shall see in Example \ref{counter}.
In general, it is a hard problem to give a precise description of the decomposition of $M(\mu')\otimes M(\mu'')$, in particular, because the tensor product is not necessarily completely reducible even when both $M(\mu')$ and $M(\mu'')$ are irreducible. \\

Let $\mathfrak{g}:=\sltwo(\mathbb{C})$. We identify $\mathfrak{h}^{*}$ with $\mathbb{C}$ such that the unique positive root is given by $2$ as usual.
Then, the Verma module $M(\mu)$ has a $(\mu+1)$-dimensional simple quotient $L(\mu)$ if $\mu \in \mathbb{N}:=\{0,1,2,\cdots \}$. Following the notation as in \cite[p. 894]{KP}, we define the \textit{multiplicity space} of $M(-\lambdathree)$ in $M(-\lambdaone)\otimes M(-\lambdatwo)$ by:
\begin{align*}
H(\lambdaone,\lambdatwo,\lambdathree):=\Hom_{\mathfrak{g}}(M(-\lambdathree),M(-\lambdaone)\otimes M(-\lambdatwo)).
\end{align*}

Kobayashi-Pevzner \cite[Thm. 9.1]{KP} discovered that the Rankin-Cohen brackets $\mathcal{RC}_{\lambdaone,\lambdatwo}^{\lambdathree}$ (see \eqref{eqn:RC} below) vanishes if and only if $H(\lambdaone,\lambdatwo,\lambdathree)$ is of dimension $2$, and found all the triples $\{\lambda',\lambda'',\lambda'''\}$ of parameters when this happens. In Theorem \ref{KPtheorem}, we give yet another characterization of this condition 
by means of the $\mathfrak{g}$-module structure of the $Z(\mathfrak{g})$-primary component  $(M(-\lambdaone)\otimes M(-\lambdatwo))^{\chi_{-\lambda'''+1}}$ corresponding to the generalized infinitesimal character $-\lambda'''+1$, see Section \ref{notation}.

\begin{thm}\label{KPtheorem}
If $\lambdathree-\lambdaone-\lambdatwo \in 2\mathbb{N}$, then the following four conditions on the triple $(\lambda', \lambda'', \lambda''')$ are equivalent:\\
(i) $\dim H(\lambdaone,\lambdatwo,\lambdathree)=2$.\\
(ii) $\lambdaone,\lambdatwo,\lambdathree \in \mathbb{Z},\ 2\geq \lambdaone+\lambdatwo+\lambdathree,\ and\ \lambdathree\geq | \lambdaone-\lambdatwo| +2$.\\
(iii) $\mathcal{RC}_{\lambdaone,\lambdatwo}^{\lambdathree}=0$.\\
(iv) $2\leq \lambdathree$ and $(M(-\lambdaone)\otimes M(-\lambdatwo))^{\chi_{-\lambda'''+1}}=M(-\lambdathree)\oplus M(\lambdathree-2)$.
\end{thm}

The equivalences $(i)\Leftrightarrow(ii)\Leftrightarrow(iii)$ were established earlier in \cite[Sect. 9]{KP}. The operator $\mathcal{RC}_{\lambdaone,\lambdatwo}^{\lambdathree}$ in $(iii)$ is a bidifferential operator, referred to as the \textit{Rankin-Cohen bracket} of degree $\ell:=\frac{1}{2}(\lambda'''-\lambda'-\lambda'')\in \mathbb{N}$, defined by 
\begin{equation}\label{eqn:RC}
\mathcal{RC}_{\lambdaone,\lambdatwo}^{\lambdathree}:=\sum_{j=0}^{\ell} (-1)^{j} \frac{(\lambdaone +\ell-1)_j (\lambdatwo +\ell -1)_{\ell-j}}{j!(\ell-j)!} \frac{\partial^\ell}{\partial {z_1}^{\ell-j} \partial {z_2}^j}\Bigg\rvert_{z_1 =z_2},
\end{equation}
where $(x)_n$ stands for the descending factorial $x(x-1)\cdots (x-n+1)$. The operator was introduced originally in number theory to produce holomorphic modular forms of higher weight $\lambdathree=\lambdaone+\lambdatwo+2\ell\ (\ell\in \mathbb{N})$ out of those of lower weights $\lambdaone$ and $\lambdatwo$. 
From the viewpoint of representation theory, $\mathcal{RC}_{\lambdaone,\lambdatwo}^{\lambdathree}$ is a \textit{symmetry breaking operator} from a $\tilde{G}\times \tilde{G}$-equivalent holomorphic line bundle to a $\tilde{G}$-equivalent one over the Poincar\'e disk where $\tilde{G}$ denotes the universal covering group of $SU(1,1)$. Such an operator shows how $\tilde{G}\times \tilde{G}$-symmetry is ``broken'' in keeping the diagonal symmetry of $\tilde{G}$. See \cite[Sect. 1.1 and Sect. 9]{KP}.\\

The novelty of Theorem \ref{KPtheorem} is the equivalence of $(ii) \Leftrightarrow (iv)$, which is derived from Theorem \ref{strtheorem} about a description of the $\mathfrak{g}$-module structure of the tensor product of two Verma modules.
Let $P(a)$ denote the projective covering of $M(a)$, see Definition \ref{projectivecover} and Example \ref{pjcoverex}.

\begin{thm}\label{strtheorem}
For $\muone,\mutwo\in \mathbb{C}$, one has an isomorphism as  $\mathfrak{g}$-modules:
\begin{equation*}
M(\muone)\otimes M(\mutwo)\cong \bigoplus_{a\in A}P(a) \oplus \bigoplus_{b\in B}M(b).
\end{equation*}
We set $X:=\muone+\mutwo-2\mathbb{N}$. The finite (possibly empty) set $A\equiv A(\muone,\mutwo)$ and a countable set $B\equiv B(\muone,\mutwo)$ are defined as follows:
\begin{align*}
\label{eqn:A}
A&:=
\begin{cases}
    \emptyset & if\ \muone+\mutwo\notin \mathbb{N},\\
    X\ \cap\ [-|\muone-\mutwo |,-2] & if\ \muone,\mutwo \in \mathbb{N},\\
    X\ \cap\ [-\muone-\mutwo-2,-2] & otherwise,
\end{cases}\\
B&:=X \setminus (A\cup A^*),
\end{align*}
where $A^*:=\{-a-2\in \mathbb{Z}:\ a\in A\}$.

\end{thm}

See \cite[Thm. 4.10]{KØSS} for analogous results on the restriction of parabolic Verma modules with respect to $\mathfrak{so}_{n+1}\supset \mathfrak{so}_{n}$. 


\begin{cor}\label{COR}
For the tensor product $M(\muone)\otimes M(\mutwo)$, the following four conditions on $(\mu',\mu'')\in \mathbb{C}^2$ are equivalent.\\
(i) $\dim \Hom_{\mathfrak{g}}(N,M(\mu')\otimes M(\mu''))\leq 1$ for any simple $\mathfrak{g}$-module $N\in \mathcal{O}$;\\
(ii) $\dim \Hom_{\mathfrak{g}}(M(\nu),M(\mu')\otimes M(\mu''))\leq 1$ for all $\nu \in \mathbb{C}$;\\
(iii) the tensor product $M(\mu')\otimes M(\mu'')$ is self-dual;\\
(iv) it does not contain any reducible Verma module as a direct summand.
\end{cor}

We shall use Remark \ref{userem} in the proof of Theorem \ref{KPtheorem}.

\begin{rem}
The finite set $A$ is contained in $\{-2,-3,-4,\cdots\}$, see Example \ref{pjcoverex}. It is non-empty if and only if $(\muone,\mutwo)\in \mathbb{C}^2$ satisfies $\muone+\mutwo\in \mathbb{N}$ and $\muone-\mutwo\notin\{0,\pm 1\}$.
\end{rem}

\begin{rem}\label{userem}
For $\nu\in \mathbb{N}$, the Verma module $M(\nu)$ occurs in $M(\muone)\otimes M(\mutwo)$ as a direct summand if and only if 
$\muone,\mutwo \in \mathbb{N}$ and 
$|\muone-\mutwo| \leq \nu\leq \muone+\mutwo$. See \cite[Thm. 9.1]{KP}. This condition is an analogue of the Clebsch-Gordan formula.
\end{rem}

\begin{exam}\label{counter}
If $(\mu',\mu'',\nu',\nu'')=(0,0,i,-i)$, then the tensor product $M(\mu')\otimes M(\mu'')$ is not isomorphic to $M(\nu')\otimes M(\nu'')$ because Theorem \ref{strtheorem} shows
\begin{align*}
M(0)\otimes M(0)&\cong \bigoplus_{k=0}^\infty M(-2k),\\
M(i)\otimes M(-i)&\cong P(-2)\oplus \bigoplus_{k=2}^\infty M(-2k),
\end{align*}
but $P(-2)$ is not isomorphic to the direct sum $M(0)\oplus M(-2)$.
In this case, the assumption in Theorem \ref{generaltheorem} is not satisfied because $0$ is not anti-dominant.
\end{exam}

\section{Proof of Theorem \ref{generaltheorem}.}

Our proof relies on the earlier work \cite{Kob} and \cite{Johan}. Kobayashi \cite{Kob} established a general framework of the restriction of parabolic Verma modules to reductive subalgebras, and proved the branching rule in the Grothendieck group for the discretely decomposable case, whereas {K\aa}hrstr\"{o}m  \cite{Johan} showed that a tensoring with Verma modules preserves Verma flags.

\subsection{Preliminaries of parabolic Verma modules.}\label{notation}
We collect some basic notions of parabolic Verma modules, see \cite{Hum}.
Let $\mathfrak{g}$ be a semisimple Lie algebra over $\mathbb{C}$, $\mathfrak{h}$ a Cartan subalgebra, and W the Weyl group of the root system $\Delta:=\Delta(\mathfrak{g},\mathfrak{h})$. We fix a positive system $\Delta^+$, take  the corresponding Borel subalgebra $\mathfrak{b}:=\mathfrak{h}+\mathfrak{n}$ with nilpotent radical $\mathfrak{n}=\oplus_{\alpha \in \Delta^{+}} \mathfrak{g}_{\alpha}$, and let $\rho$ be half the sum of positive roots. A standard parabolic subalgebra $\mathfrak{p}$ is a parabolic subalgebra containing the Borel subalgebra  
$\mathfrak{b}$. We take a Levi decomposition $\mathfrak{p}=\mathfrak{l}+\mathfrak{u}$ with $\mathfrak{h}\subset \mathfrak{l}$, and set $\Delta^{+}(\mathfrak{l}):=\Delta^{+}\cap \Delta(\mathfrak{l},\mathfrak{h})$ and $\mathfrak{n}_{-}(\mathfrak{l}):=\oplus_{\alpha \in \Delta^+ (\mathfrak{l})} \mathfrak{g}_{-\alpha}$. Let $U(\mathfrak{g})$ be the universal enveloping algebra of $\mathfrak{g}$, and $Z(\mathfrak{g})$ its center.\\

A $\mathfrak{g}$-module $M$ is called \textit{locally $Z(\mathfrak{g})$-finite} if for every $x\in M$ the space $Z(\mathfrak{g})x$ is finite dimensional. Let $\mathcal{Z}_{fin}$ be the category of locally $Z(\mathfrak{g})$-finite $\mathfrak{g}$-modules. Category $\mathcal{O}$ is the full subcategory of $\mathcal{Z}_{fin}$ such that the objects are finitely generated and $\mathfrak{h}$-semisimple, and locally $\mathfrak{n}$-finite. Category $\mathcal{O}^{\mathfrak{p}}$ is the full subcategory of $\mathcal{O}$ such that the objects are locally $\mathfrak{n}_{-}(\mathfrak{l})$-finite.\\

We write the set of dominant integral for $\Lambda^{+}(\mathfrak{l}):=\{\lambda\in \mathfrak{h}^{*}: \langle\lambda,\alpha^{\vee}\rangle\in \mathbb{N}\quad \text{for all}\ \alpha \in \Delta^{+}(\mathfrak{l}) \}$ where $\alpha^{\vee}$ denotes the coroot of $\alpha$. For $\lambda \in \Lambda^{+} (\mathfrak{l})$, let $F_\lambda$ denote the finite-dimensional irreducible $\mathfrak{l}$-module with highest weight $\lambda$. We also write $\mathbb{C}_\lambda$ for $F_\lambda$ if it is one-dimensional. We regard $F_\lambda$ (and $\mathbb{C}_\lambda$) as a $\mathfrak{p}$-module by letting $\mathfrak{u}$ act trivially. The $\mathfrak{g}$-module $M^{\mathfrak{g}}_{\mathfrak{p}}(\lambda):=U(\mathfrak{g})\otimes_{U(\mathfrak{p})} F_\lambda $ is referred to as a parabolic Verma module. For $\mathfrak{p}=\mathfrak{b}$, it reduces to the usual Verma module $M(\lambda)\equiv M^{\mathfrak{g}}_{\mathfrak{b}}(\lambda)=U(\mathfrak{g})\otimes_{U(\mathfrak{b})} \mathbb{C}_{\lambda}$ for $\lambda\in \mathfrak{h}^*$. An element $\lambda \in \mathfrak{h}^{*}$ is said to be \textit{anti-dominant} if $\langle \lambda+\rho , \alpha^{\vee} \rangle \notin \mathbb{N}_{+}=\{1,2,\cdots\}$ for all $\alpha \in \Delta^+$, equivalently, it $M(\lambda)$ is simple. \\

The Harish-Chandra isomorphism parametrizes the set $\Hom_{\mathbb{C}\text{-alg}}(Z(\mathfrak{g}),\mathbb{C})$ of $\mathbb{C}$-algebraic homomorphisms as
\begin{align*}
\mathfrak{h}^{*}/W \cong \Hom_{\mathbb{C}\text{-alg}}(Z(\mathfrak{g}),\mathbb{C}),\ \nu \mapsto \chi_{\nu}.
\end{align*}

In our parametrization, the infinitesimal character of the trivial module is $\rho \in \mathfrak{h}^{*}/W$ and that of the Verma module $M(\lambda)$ is $\lambda+\rho$. 
Any locally $Z(\mathfrak{g})$-finite $\mathfrak{g}$-module $M$ admits a primary decomposition
\begin{align*}
    M=\bigoplus_{\nu\in \mathfrak{h}*/W} M^{\chi_{\nu}},
\end{align*}
where $M^{\chi_\nu}$ is the $\nu$-primary component of $\nu$, namely, $M^{\chi_\nu}$ is the $\mathfrak{g}$-submodule of elements annihilated by a power of $\Ker (\chi_\nu : Z(\mathfrak{g})\rightarrow \mathbb{C})$.

\begin{dfn}[{\cite[Ch.\ 3.9]{Hum}}]\label{projectivecover}
For $M\in \mathcal{O}$, a surjective $\mathfrak{g}$-map $\phi:P\to M$ is called a \textit{projective covering} of $M$ if $P\in \mathcal{O}$ is a projective module without any proper submodule such that its image of $\phi$ equals $M$.
\end{dfn}

We omit $\phi$ when it is obvious. For $\lambda\in \mathfrak{h}^*$, let $P(\lambda)$ denote the projective cover of the Verma module $M(\lambda)$.

\begin{exam}\label{pjcoverex}
Let $\mathfrak{g}=\sltwo(\mathbb{C})$. For $\lambda\notin\{-2,-3,\cdots\}$, one has $P(\lambda ) =M(\lambda)$. If $\lambda \in \{-2,-3,\cdots \}$, there is a non-split short exact sequence of $\mathfrak{g}$-modules:
\begin{align*}
0 \rightarrow M(-\lambda-2) \rightarrow P(\lambda) \rightarrow M(\lambda)\rightarrow 0 .
\end{align*}
\end{exam}

\subsection{Branching problem for parabolic Verma modules.}

The tensor product $M(\mu')\otimes M(\mu'')$ may be viewed as the restriction of the Verma module $M(\mu',\mu'')=U(\mathfrak{g}\oplus \mathfrak{g})\otimes_{U(\mathfrak{b}\oplus \mathfrak{b})}\mathbb{C}_{(\mu',\mu'')}$ to the subalgebra $\diag(\mathfrak{g})$ of $\mathfrak{g}\oplus \mathfrak{g}$. We then utilize the general framework in \cite{K98} and \cite{Kob}. In particular, the concept of discrete decomposability of $\mathfrak{g}$-modules, which was introduced in \cite{K98} as an algebraic counterpart of the unitary representation without continuous spectrum, is useful in our study of the tensor product.

\begin{dfn}[{\cite[Def. 3.1]{Kob}}]
A $\mathfrak{g}$-module $X$ is called \textit{discretely decomposable} if there is an ascending sequence
$\{ X_{i} \}_{i\in \mathbb{N}}$ of  submodules such that $\bigcup X_{i} = X$ and that each $X_{i}$ has finite length.\\
Moreover for a full subcategory $\mathcal{A}$ of the category of $\mathfrak{g}$-modules, we say $X\in \mathcal{A}$ is \textit{discretely decomposable} in $\mathcal{A}$ if every $X_i$ belongs to $\mathcal{A}$.
\end{dfn}

Consider the restriction of a generalized Verma module $M^{\mathfrak{g}}_{\mathfrak{p}}(\mu)$ with respect to a reductive subalgebra $\mathfrak{g}'$ of $\mathfrak{g}$. In general, the restriction may not contain any simple $\mathfrak{g}'$-module. A geometric criterion on the pair $(\mathfrak{g}',\mathfrak{p})$ was proved in \cite{Kob} about when the restriction $M^{\mathfrak{g}}_{\mathfrak{p}}(\mu)\rvert_{\mathfrak{g}'}$ is descretely deconposable. In what follows, we recall an \textit{algebraic condition} which is slightly stronger than this \textit{geometric condition}. \\
A semisimple element $H\in \mathfrak{h}$ is said to be hyperbolic if $\ad(H)$ is diagonalizable with all eigenvalues being real. 
We write $\mathfrak{g}=\mathfrak{u}_{-}+ \mathfrak{l} +\mathfrak{u}$ for the decomposition into the sum of eigenspaces of $\ad(H)$ with negative, $0$ and positive eigenvalues, respectively.
Then $\mathfrak{p}(H):=\mathfrak{l}+ \mathfrak{u}$ is a parabolic subalgebra of $\mathfrak{g}$. 

\begin{dfn}[{\cite[Def. 3.7]{Kob}}]\label{Kobdef} A parabolic subalgebra $\mathfrak{p}$ is said to be $\mathfrak{g}'$-$compatible$ if there exists a hyperbolic element $H$ in $\mathfrak{g}'$ such that $\mathfrak{p}=\mathfrak{p}(H)$.
\end{dfn}

If $\mathfrak{p}=\mathfrak{l}+\mathfrak{u}$ is $\mathfrak{g}'$-compatible, then $\mathfrak{p}':=\mathfrak{p}\cap \mathfrak{g}'$ becomes a parabolic subalgebra of $\mathfrak{g}'$ with the Levi decomposition:

\begin{equation}\label{eqn:smallpara}
\mathfrak{p}'=\mathfrak{l}'+\mathfrak{u}'
:=(\mathfrak{l}\cap\mathfrak{g}')+(\mathfrak{u}\cap \mathfrak{g}').
\end{equation}

We use the following fact that guarantees the discrete decomposability of the restriction.

\begin{fact}[{\cite[Prop. 3.8]{Kob}}]\label{disdec}
If a parabolic subalgebra $\mathfrak{p}$ of $\mathfrak{g}$ is $\mathfrak{g}'$-$compatible$, then the restriction $X\rvert_{\mathfrak{g}'}$ is discretely decomposable in $\mathcal{O}^{\mathfrak{p}'}$ for any $X\in \mathcal{O}^\mathfrak{p}$
\end{fact}

For a $\mathfrak{g}'$-compatible parabolic subalgebra $\mathfrak{p}$, each eigenspace of $H$ on $M^{\mathfrak{g}}_{\mathfrak{p}}(\mu) \cong U(\mathfrak{u}_{-})\otimes F_\mu$ is finite-dimensional, hence one has:

\begin{rem}\label{fdrem}
The weight spaces of $M^{\mathfrak{g}}_{\mathfrak{p}}(\mu) \rvert_{\mathfrak{g}'}$ are all finite-dimensional.
\end{rem}

Given a vector space V, we denote by $S(V):=\oplus_{k=0}^{\infty} S^{k} (V)$ the symmetric tensor algebra. Retain the notation as in $\eqref{eqn:smallpara}$. We extend the adjoint action of $\mathfrak{l}'$ on $\mathfrak{u}_{-}/\mathfrak{u}_{-}\cap \mathfrak{g}'$ to 
$S(\mathfrak{u}_{-}/\mathfrak{u}_{-}\cap \mathfrak{g}^{'})$.
As in \cite[Sect. 3.3]{Kob}, we take a Cartan subalgebra $\mathfrak{h}'$ of $\mathfrak{g}'$ such that $H\in \mathfrak{h}'$, and extend it to a Cartan subalgebra $\mathfrak{h}$ of $\mathfrak{g}$. Clearly $\mathfrak{h}\subset \mathfrak{l}$ and $\mathfrak{h}'\subset \mathfrak{l}'$.
For $\delta \in \Lambda^{+}(\mathfrak{l}')$, we denote by $F_{\delta}'$ the simple $\mathfrak{l}'$-module of highest weight $\delta$. We set

\begin{align*}
m(\delta;\mu):=\dim \Hom_{\mathfrak{l}'}(F'_{\delta},F_{\mu}\otimes S(\mathfrak{u}_{-}/\mathfrak{u}_{-}\cap \mathfrak{g}')).
\end{align*}

\begin{fact}[{\cite[Thm. 3.10]{Kob}}]\label{Kobthm}
Suppose that $\mathfrak{p}=\mathfrak{l}+\mathfrak{u}$ is $\mathfrak{g}'$-compatible. Then for any $\mu \in \Lambda^{+} (\mathfrak{l})$,\\
(1) $m(\delta;\mu)<\infty\ for\ all\ \delta\in \Lambda^{+} (\mathfrak{l'})$.\\
(2) In\ the\ Grothendieck\ group, $[M^{\mathfrak{g}}_{\mathfrak{p}}(\mu) \rvert_{\mathfrak{g}'}]=\bigoplus_{\delta \in \Lambda^{+} (\mathfrak{l'})} m(\delta;\mu)\ [M^{\mathfrak{g}'}_{\mathfrak{p}'}(\delta)]$.
\end{fact}

\begin{exam}
We write $\Gamma$ for the semigroup generated by positive roots, and $M(\mu)_\nu$ for the $\nu$-weight space of $M(\mu)$. Applying Fact \ref{Kobthm} to $(\mathfrak{g}\oplus \mathfrak{g},\diag(\mathfrak{g}))$, we get the following equation in the Grothendieck group:

\begin{align*}
[M(\mu')\otimes M(\mu'')]=\bigoplus_{\nu \in \mu'+\mu''-\Gamma} \dim M(\mu'+\mu'')_{\nu} \ [M(\nu)].
\end{align*}    
\end{exam}

We summarize some direct consequences of Fact \ref{disdec}, Remark \ref{fdrem}, and Fact \ref{Kobthm}.

\begin{cor}
Retain the setting of Fact \ref{Kobthm}. For any 
$M\in \mathcal{O}^{\mathfrak{p}}$, in particular for $M=M^{\mathfrak{g}}_{\mathfrak{p}}(\mu) \rvert_{\mathfrak{g}'}$, the restriction $M \rvert_{\mathfrak{g}'}$\ admits a $Z(\mathfrak{g}')$-primary decomposition and each component belongs to
$\mathcal{O}^{\mathfrak{p}'}$.
\end{cor}

We then reduce a question to find the branching rules to some discussion in BGG category $\mathcal{O}$.

\subsection{Tensor products and Verma flags.}
We recall from \cite{Hum} and \cite{Johan} the notion of Verma flags and some compatibility of the tensor product of two modules in $\mathcal{O}$ and Verma flags.

\begin{dfn}[{\cite{Johan}}]
Let $\Oinf$ be the full subcategory of $\mathcal{Z}_{fin}$ of which objects $M$ have the property that there are finitely many weights $\lambda_{1}, \cdots \lambda_{n}$ such that 
$\Pi (M)\subset \bigcup (\lambda_{i}-\Gamma)$, where we denote by $\Pi(M)$ the set of weights of $M$.

\end{dfn}


\begin{rem}
For the symmetric pair $(\mathfrak{g}\oplus \mathfrak{g},\diag(\mathfrak{g}))$, the Borel subalgebra $(\mathfrak{b}\oplus \mathfrak{b})$ is $\diag(\mathfrak{g})$-compatible, see Definition \ref{Kobdef}. Thus it follows from Fact \ref{disdec} that for any $M,N\in \mathcal{O}$, the tensor product $M\otimes N \in \Oinf$ and admits a primary decomposition. 
\end{rem}

Let us review dual modules in $\mathcal{O}$ and the Verma flags, see \cite[Ch. 3]{Hum}.\\
Fix an anti-involution $\tau$ of a complex semisimple Lie algebra $\mathfrak{g}$ such that $\tau$ switches $\mathfrak{g}_{\alpha}$ with $\mathfrak{g}_{-\alpha}$ and is the identity on $\mathfrak{h}$. For $M\in \mathcal{Z}_{fin}$, we define a $\mathfrak{g}$-module on
\begin{align*}
M^{\vee}:=\bigoplus_{\lambda\in \mathfrak{h}^*} \Hom_{\mathbb{C}} (M_{\lambda},\mathbb{C})
\end{align*}
by $(X.f)(v):=f(\tau(X).v)$ for $f\in M^{\vee}$, $v\in M$, and $X\in \mathfrak{g}$. For $M,N\in \mathcal{O}$, one has $(M\otimes N)^{\vee} \cong M^{\vee} \otimes N^{\vee}$, see \cite{Johan} for example.\\

We say $M\in \mathcal{Z}_{fin}$ has a \textit{Verma flag} or \textit{standard filtration} if there is an ascending sequence of submodules $\{M_{i}\}$ such that $\bigcup M_{i}=M$, $M_0=0$, and $M_{i+1}/M_{i}$ are isomorphic to some Verma module for every $i$.
If $M\in \mathcal{O}$, ${M_{i}}$ is stable for sufficiently large $i$.\\

We write $M \in \bigtriangleup$ if $M$ has a Verma flag, $M \in \bigtriangledown$ if $M^{\vee}$ has a Verma flag, and $M\in \bigtriangleup \cap \bigtriangledown$ if $M$ and $M^\vee$ have a Verma flag.

\begin{dfn}[{\cite[Ch. 11]{Hum}}]
    $X\in \mathcal{O}$ is called a \textit{tilting module} if $M\in \bigtriangleup \cap \bigtriangledown$.\\
    $X\in \Oinf$ is called a \textit{tilting module} if all primary components are tilting modules.
\end{dfn}

\begin{fact}[{\cite[Prop. 2.6]{Johan}}]\label{Johanthm}
Let $M,N\in \mathcal{O}$. For $\mathcal{F} \in \{\bigtriangleup,\bigtriangledown,\bigtriangleup \cap \bigtriangledown\}$, one has
\begin{align*}
    M\in \mathcal{F} \Rightarrow M\otimes N \in \mathcal{F}.
\end{align*}
\end{fact}

It is known, see \cite[Ch. 11]{Hum}, that the $\mathfrak{g}$-module structure of a tilting module is determined by its Grothendieck group.

\subsection{Proof of Theorem \ref{generaltheorem}.}

We are ready to complete the proof of Theorem \ref{generaltheorem}.

\begin{prop}\label{eqgro}
If $\mu \in \mathfrak{h}^{*}$ and $N\in \mathcal{O}$, then one has the following isomorphism in the Grothendieck group:
\begin{align*}
[M(\mu)\otimes N]=\bigoplus_{\nu\in \Pi(N)} \dim N_{\nu}\  [M(\mu+\nu)].
\end{align*}
\end{prop}

\begin{proof}
Since $N$ is a $\mathfrak{g}$-module, one has the following isomorphism of $\mathfrak{g}$-modules.
\begin{align*}
M(\mu)\otimes N =(U(\mathfrak{g})\otimes_{U(\mathfrak{b})} \mathbb{C}_{\mathfrak{\mu}}) \otimes N \cong U(\mathfrak{g}) \otimes_{U(\mathfrak{b})} (\mathbb{C}_{\mu} \otimes N).
\end{align*} 
We take a filtration $0=B_0\subset B_1\subset B_2\subset \cdots$ of $\mathbb{C}_{\mu} \otimes N$ as a $\mathfrak{b}$-module such that $B_{i+1}/B_i$ is isomorphic to $\mathbb{C}_{\mu_i}$ for some $\mu_i \in \mu+\Pi(N)$. In turn, $\{ U(\mathfrak{g}) \otimes_{U(\mathfrak{b})} B_i \}$ is a Verma flag of $M(\mu)\otimes N$.
Then the proposition follows because $(U(\mathfrak{g})\otimes_{U(\mathfrak{b})} B_{i+1})/(U(\mathfrak{g})\otimes_{U(\mathfrak{b})} B_i) \cong M(\mu_{i})$.
\end{proof}

\begin{rem}
If $N$ is a Verma module, this is a special case of \cite[Thm. 3.10]{Kob}.
\end{rem}

As a consequence of Proposition \ref{eqgro}, we get the following corollary.

\begin{cor}\label{grocor}
Suppose that $\mu',\mu'',\nu',\nu'' \in \mathfrak{h}^{*}$ satisfy $\mu'+\mu''=\nu'+\nu''$, then one has an isomorphism in the Grothendieck group:
\begin{align*}
[M(\mu')\otimes M(\mu'')]=[M(\nu')\otimes M(\nu'')].
\end{align*}
\end{cor}

Under the assumption of Theorem $\ref{generaltheorem}$, at least one of $M(\mu')$ and $M(\mu'')$ is simple. Similarly for $M(\nu')$ and $M(\nu'')$. Therefore we obtain Theorem \ref{generaltheorem} since simple Verma modules belong to $\bigtriangleup \cap \bigtriangledown$.

\section{Indecomposable components of $M(\muone)\otimes M(\mutwo)$ in the $\sltwo$ case.}

Throughout this section, we suppose $\mathfrak{g}$ is $\sltwo(\mathbb{C})$. We prove Theorem \ref{strtheorem} and Corollary \ref{COR}.

Since each weight space of Verma modules is one-dimensional for $\sltwo(\mathbb{C})$, the following equation holds in the Grothendieck group from Proposition \ref{eqgro}.

\begin{align}\label{eqn:C}
[M(\muone)\otimes M(\mutwo)]=\bigoplus_{k\in \mathbb{N}}  [M(\muone+\mutwo-2k)].
\end{align}

Let $M^{\chi_{\nu+1}}$ denote the $Z(\mathfrak{g})$-primary component of the tensor product $M(\muone)\otimes M(\mutwo)$ for the generalized eigenvalue $\nu+1\in \mathfrak{h}^*/W \cong \Hom_{\mathbb{C}\text{-alg}}(Z(\mathfrak{g}),\mathbb{C})$. Then one has a primary decomposition:

\begin{align*}
M(\muone)\otimes M(\mutwo)\cong\bigoplus_{\nu+1 \in W\cdot (\muone+\mutwo+1-2\mathbb{N})} M^{\chi_{\nu+1}}.
\end{align*}
\noindent
We know that the $\mathfrak{g}$-module $M^{\chi_{\nu+1}}$ is of finite length, and want to determine the structure of each $M^{\chi_{\nu+1}}$ for the proof of Theorem \ref{strtheorem}.


\begin{lem}\label{uselem}
For every $\nu\in X$, $M^{\chi_{\nu+1}}$ has a Verma flag. Moreover, in the Grothendieck group, one has
\begin{align*}
[M^{\chi_{\nu+1}}]=
\begin{cases}
[M(\nu)+M(-\nu-2)] &  \nu \in A'\cup A'^*,\\
[M(\nu)] & \nu \notin A'\cup A'^*.
\end{cases}
\end{align*}

\end{lem}

\begin{proof}

Since $M(\muone)\otimes M(\mutwo)$ has a Verma flag, so does every component $M^{\chi_{\nu+1}}$, whence the first assertion. The second assertion is derived from $\eqref{eqn:C}$ because
\begin{align*}
    X\cap X^*\setminus \{-1\} = A'\cup A'^*
\end{align*} is the set of $\nu\in X$ such that $\nu\neq \nu^*$ and that $[M(\nu^*)]$ occurs in $\eqref{eqn:C}$.
\end{proof}

We also use the following lemma \cite[Lem. 9.4]{KP}:

\begin{lem}\label{uselemtwo}
    For all $k\in \mathbb{N}$ and for all $\mu',\mu''\in \mathbb{C}$,
    \begin{align*}
        \Hom_{\mathfrak{g}}(M(\mu'+\mu''-2k),M(\mu')\otimes M(\mu''))\neq \{0\}.
    \end{align*}
\end{lem}

For the sake of completeness, we give a direct proof. The dimension of weight spaces is given as
\begin{align*}
    \dim (M(\muone)\otimes M(\mutwo))_{\muone +\mutwo-2k}=k+1.
\end{align*}
\noindent
Since $\diag(\mathfrak{n})$ maps $(M(\muone)\otimes M(\mutwo))_{\muone +\mutwo-2k}$ to $(M(\muone)\otimes M(\mutwo))_{\muone +\mutwo-2k+\nu}$, it has a non-trivial kernel, namely, there exists a non-zero singular vector in $(M(\muone)\otimes M(\mutwo))_{\muone +\mutwo-2k}$.\\

By Lemma \ref{uselem} and \ref{uselemtwo}, one has an isomorphism $M^{\chi_{\nu+1}} \cong M(\nu)$ for $\nu \in X\setminus (A'\cup A'^*)$. We note that $\chi_{\nu+1}=\chi_{\nu^* +1}$ where $\nu^*=-\nu-2$. What remains is to determine the structure of $M^{\chi_{\nu+1}}$ for $\nu \in A'$. We compute $M^{\chi_{\nu+1}}$ for $\nu\in A'$ in by case-by-case argument as below.

\subsection{Generic case.}\label{genericcase}

We begin with the case $\muone +\mutwo \notin \mathbb{N}$. Every Verma module $M(\muone+\mutwo-2k)$ is irreducible, and the set of (generalized) eigenvalues $\{\mu'+\mu''-2k+1 : k\in \mathbb{N}\}$ are distinct from each other. Hence the equation of the Grothendieck group \eqref{eqn:C} is the isomorphism of modules.

\subsection{Singular case 1.}

We assume $\muone+\mutwo \in \mathbb{N}$ but at least one of $\mu'$ or $\mu''$ does not belong to $\mathbb{N}$. We observe that $A'$ coincides with $A(\mu',\mu'')$ in Theorem \ref{strtheorem}. On the other hand, $M(\muone)\otimes M(\mutwo)$ is a tilting module by Fact \ref{Johanthm}.
Suppose $\nu\in A'$. Since $[M^{\chi_{\nu+1}}]=[M(\nu)]+[M(-\nu-2)]=[P(\nu)]$ by Example \ref{pjcoverex} and Lemma \ref{uselem}, we conclude
$M^{\chi_{\nu+1}} \cong P(\nu)$ because both the primary component $M^{\chi_{\nu+1}}$ and the projective covering $P(\nu)$ are tilting modules. Hence Theorem \ref{strtheorem} is shown in this case.

\subsection{Singular case 2.}

Finally, we assume $\muone,\mutwo \in \mathbb{N}$.
Without loss of generality, we may assume $\muone \geq \mutwo$.
One has the following short exact sequence:
\begin{align*}
0\rightarrow M(\muone)\otimes M(-\mutwo-2) \rightarrow M(\muone)\otimes M(\mutwo)\rightarrow M(\muone)\otimes L(\mutwo) \rightarrow 0.
\end{align*}
We have already determined the structure of $M(\muone)\otimes M(-\mutwo-2)$ in the previous cases. We note that the set $A(\mu',\mu'')$ in Theorem \ref{strtheorem} is equal to $A'(\mu',-\mu''-2)$ since $-\muone-(-\mutwo-2)-2=-|\muone-\mutwo |$. Moreover, since $M(\muone)\otimes L(\mutwo)$ is a projective module, the above short exact sequence splits. In view of a Verma flag of $M(\muone)\otimes L(\mutwo)$ one has the following direct sum decomposition
\begin{align*}
M(\muone)\otimes L(\mutwo) \cong \bigoplus_{k=0}^{\mutwo} M(\muone+\mutwo -2k)
\end{align*}
\noindent
because the $Z(\mathfrak{g})$-infinitesimal characters in the summands are all distinct.
Thus the proof of Theorem \ref{strtheorem} is completed.


\subsection{Proof of the equivalence (ii)$\Leftrightarrow$(iv) in Theorem \ref{KPtheorem}.}

Suppose $\Resub \nu \leq 0$. By Theorem \ref{strtheorem} and its proof,
\begin{align*}
(M(\nu’)\otimes M(\nu^{\prime \prime}))^{\chi_{\nu+1}}=M(\nu)\oplus M(-\nu-2)
\end{align*}
If and only if $\nu \in A’\setminus A$. Moreover,
\begin{align*}
A’\setminus A=\{-|\mu’-\mu^{\prime \prime}|-2j-2: 0\leq j\leq \min(\mu’.\mu^{\prime \prime}) \}
\end{align*}
If $\mu’,\mu^{\prime \prime} \in \mathbb{N}$, and $A’\setminus A=\emptyset$ otherwise. Hence the equivalence (ii)$\Leftrightarrow$(iv) in Theorem \ref{KPtheorem} is shown by putting $\lambda^{\prime \prime \prime}=-\nu$.

\setcounter{thm}{0}
\setcounter{section}{3}

\begin{exam}
We cannot drop the condition $2\leq \lambdathree$ in (iv) in Theorem \ref{KPtheorem}.
For instance, take $(\lambdaone,\lambdatwo,\lambdathree)$ to $(-1,-3,-2)$. From Theorem \ref{strtheorem}, we have
\begin{align*}
(M(1)\otimes M(3))^{\chi_{2+1}}=M(2)\oplus M(-4),
\end{align*}
but $-2< |-1-(-3)| +2$. It does not satisfy (ii).
\end{exam}

\subsection{Proof of Corollary \ref{COR}}

\begin{proof}
    The implication (i)$\Rightarrow$(ii) is obvious. We recall $B\equiv B(\mu',\mu'')$ from Theorem \ref{strtheorem}, which tells us that the condition (iii) is equivalent to $B\cap \mathbb{N}=\emptyset$ because the projective covering $P(a)$ and irreducible $M(b)$ are self-dual. On the other hand, $B\cap \mathbb{N}\neq \emptyset$ if and only if $\mu',\mu'' \in \mathbb{N}$, and in this case both $\nu$ and $\nu^*=-\nu-2$ belong to 
    \begin{align*}
        B\cap \mathbb{N}&=\{|\mu'-\mu''|+2j:0\leq j\leq \min(\mu',\mu'') \}\\
        &=\{-\nu-2:\nu\in A'\setminus A\}.
    \end{align*}
    By Theorem \ref{strtheorem}, $\dim \Hom (M(-\nu-2),M(\mu')\otimes M(\mu''))=\dim \Hom(M(-\nu-2),M(-\nu-2)\oplus M(\nu))=2$ for any $\nu \in B\cap \mathbb{N}$. Hence the implication (ii)$\Rightarrow$(iii) is shown. On the other hand, it follows from Theorem \ref{strtheorem} that indecomposable summand of $M(\mu')\otimes M(\mu'')$ are given as
    \begin{align*}
        \bigoplus_{a\in A} P(a) \oplus \bigoplus_{b\in B\cap \mathbb{N}} M(b) \oplus \bigoplus_{b\in B\setminus \mathbb{N}} M(b).
    \end{align*}
    Since $\{b^*:b\in B\cap\mathbb{N}\}\subset B\setminus \mathbb{N}$, this direct sum is multiplicity-free if and only if $B\cap \mathbb{N}=\emptyset$, hence one has the equivalence (i)$\Leftrightarrow$(iii). By Theorem \ref{strtheorem}, $M(\mu')\otimes M(\mu'')$ contains a reducible Verma module as a direct summand if and only if $B\cap \mathbb{N}=\emptyset$, hence (iii)$\Leftrightarrow$(iv) is shown.
\end{proof}




\nocite{KØSS}

\bibliography{sample} 

\providecommand{\bysame}{\leavevmode\hbox to3em{\hrulefill}\thinspace}
\providecommand{\MR}{\relax\ifhmode\unskip\space\fi MR }
\providecommand{\MRhref}[2]{%
  \href{http://www.ams.org/mathscinet-getitem?mr=#1}{#2}
}
\providecommand{\href}[2]{#2}
\begin{thebibliography}{1}

\bibitem{Hum}
James~E. Humphreys, \emph{Representations of semisimple {L}ie algebras in the {BGG} category {$\mathscr{O}$}}, Graduate Studies in Mathematics, vol.~94, American Mathematical Society, Providence, RI, 2008. \MR{2428237}

\bibitem{Johan}
Johan {K\aa}hrstr\"{o}m, \emph{Tensoring with infinite-dimensional modules in {$\mathscr O_0$}}, Algebr. Represent. Theory \textbf{13} (2010), no.~5, 561--587. \MR{2684221}

\bibitem{K98}
Toshiyuki Kobayashi, \emph{Discrete decomposability of the restriction of {$A_{\mathfrak{q}}(\lambda)$} with respect to reductive subgroups. {III}. {R}estriction of {H}arish-{C}handra modules and associated varieties}, Invent. Math. \textbf{131} (1998), no.~2, 229--256. \MR{1608642}

\bibitem{Kob}
\bysame, \emph{Restrictions of generalized {V}erma modules to symmetric pairs}, Transform. Groups \textbf{17} (2012), no.~2, 523--546. \MR{2921076}

\bibitem{KØSS}
Toshiyuki Kobayashi, Bent {\O}rsted, Petr Somberg, and Vladim\'{\i}r Sou\v{c}ek, \emph{Branching laws for {V}erma modules and applications in parabolic geometry. {I}}, Adv. Math. \textbf{285} (2015), 1796--1852. \MR{3406542}

\bibitem{KP}
Toshiyuki Kobayashi and Michael Pevzner, \emph{Differential symmetry breaking operators: {II}. {R}ankin-{C}ohen operators for symmetric pairs}, Selecta Math. (N.S.) \textbf{22} (2016), no.~2, 847--911. \MR{3477337}

\end{thebibliography}
\bibliographystyle{amsplain} 

\end{document}